\documentclass[12pt,nopagetitle]{article}
\usepackage{amsfonts, amsmath, amssymb, amsthm, enumitem}  
\usepackage[english]{babel}
\usepackage[utf8]{inputenc}
\usepackage[OT1]{fontenc}
\usepackage{bm}
\usepackage{tikz}

\usepackage{textcomp, cmap, comment}	
\usepackage{graphicx, wrapfig}
\usepackage{epigraph, xcolor, hyperref}
\usepackage{array,tabularx,tabulary,booktabs}
\usepackage{euscript, mathrsfs}	

\newcounter{iii}

\newcommand{\R}{{\mathbb R}}

\theoremstyle{plain}
\newtheorem{thm}{Theorem}

\newtheorem{lem}{Lemma}

\newtheorem{prop}{Proposition}

\newtheorem{conj}{Conjecture}

\theoremstyle{definition}
\newtheorem{defi}{Definition}

\def\Vol{\mbox{\rm Vol }}

\title{Non-dissective coverings by planks}
\author{Andrey Kupavskii\thanks{Moscow Institute of Physics and Technology, Saint-Petersburg State University, Innopolis University, Russia; Email: {\tt kupavskii@yandex.ru}} \and J\'anos Pach\thanks{R\'enyi Insitute, Hungary; Email:
{\tt pach@renyi.hu}. Supported by NKFIH grants K-131529,  and ERC Advanced Grant 882971``GeoScape.''}}

\begin{document}

\maketitle

\begin{abstract}
A \emph{plank} is the part of space between two parallel planes. The following open problem, posed 45 years ago, can be viwed as the converse of Tarski's plank problem (Bang's theorem): Is it true that if the total width of a collection of planks is sufficiently large, then the planks can be individually translated to cover a unit ball $B$?

A translative covering of $B$ by planks is said to be \emph{non-dissective} if the planks can be added one by one, in some order, such that the uncovered part remains connected at each step, and is empty at the end. Improving a classical result of Groemer, we show that every set of $C/\epsilon^{7/4}$ planks of width $\epsilon$ admits a non-dissective translative covering of $B$, provided $C$ is large enough. Our proof yields a low-complexity algorithm. We also establish the first nontrivial lower bound of $c/\epsilon^{4/3}$ for this quantity.
\end{abstract}

\section{Introduction}

A {\it plank of width $\epsilon$} in $\R^d$ is region $P$ bounded by two parallel hyperplanes at distance $\epsilon$ from each other. The distance between these two hyperplanes, is called the {\it width} of $P$.

In 1932, Tarski~\cite{Tar32} posed the question now known as the \emph{plank problem}: if a convex body $K$ in $\mathbb{R}^d$ is covered by a collection of planks, is it true that the sum of their widths is at least as large as the minimal \emph{width} of $K$, that is, the minimum width of a plank that contains $K$? Twenty years later, this question was answered affirmatively by Bang \cite{Ban50, Ban51}; see also Fenchel \cite{Fen51} for an alternative proof. Bang also conjectured that a similar statement is true with ``relative widths'' with respect to $K$. This was proved by Ball~\cite{Bal91} for centrally symmetric bodies, but in the general case the problem is open. A spherical analogue of the plank problem, raised by L.~Fejes T\'oth~\cite{Fe73}, was settled by Jiang and Polyanskii~\cite{JiP17} (see also Ortega-Moreno~\cite{Ort21}). A.~Bezdek \cite{Bez03} conjectured that the following planar variant of Tarski's conjecture is also true: the total width of a set of planks that cover a unit disk with a small circular hole around its center is at least 2;  see also \cite{WhWi07, Ambr25}.  For many other related open problems, consult~\cite{BrMP05}, Section 3.4. For a recent survey on the topic, see~\cite{Verr25}.
\medskip

The following natural question was asked independently by Groemer~\cite{Gr81} and Makai-Pach~\cite{MaP83}. Can one reverse Bang's theorem? Is it true that if the total width of a collection of planks $P_1, P_2, \ldots$ in $\mathbb{R}^d$ is large enough, then each $P_i$ can be translated into a new position $P'_i$ such that $P'_1, P'_2,\ldots$ together cover, say, a ball $B^d$ of \emph{unit radius}? In other words, does there exist a constant $C_d$ such that any collection of planks of total width at least $C_d$ permits a \emph{translative covering} of the unit ball $B^d$. In the plane, this was proved independently by Makai-Pach~\cite{MaP83} and Erd\H os-Straus (unpublished); see also~\cite{Gr82,Gr84}. For $d\ge 3$, the problem is open. In particular, we have the following.

\begin{conj} (\cite{MaP83})
For every $d\ge 3$, there exists a constant $C_d$ such that every collection $C_d\epsilon^{-1}$ planks of width $\epsilon$ in $\R^d$ permits a translative covering of the unit ball. \label{conj1}
\end{conj}

For a long time, the best known result was due to Groemer~\cite{Gr81}, who designed a greedy algorithm to construct a translative covering of the unit ball for any collection of $O(\epsilon^{-(d+1)/2})$ planks of width $\epsilon$ in $\R^d$. Kupavskii and Pach~\cite{KuP16, KuP17} came close to proving Conjecture~\ref{conj1}: they showed that the statement is true for any system of only $O(\epsilon^{-1}\log(\epsilon^{-1}))$ planks.

Nevertheless, Groemer's algorithm had a very interesting additional feature. Given a sequence of planks, $P_1,\ldots,P_f$ by a greedy method, he found a suitable translate $P'_i$ of $P_i,\; i=1,\ldots,f$ with the property that
$$B^d\setminus \cup_{i=1}^j P_i' $$
is a \emph{connected} set, for every $j\ge 1$. Moreover, Groemer showed that if we insist on this property and we also fix the order of the planks, then the bound $O(\epsilon^{-(d+1)/2})$ is essentially best possible.

\begin{defi}
We say that a collection of planks $P_1,\ldots, P_f$ permits a {\it non-dissective translative covering} of a convex body $K$ in $\R^d$ if there exist a permutation
$$\sigma: \{1,\ldots,f\}\to \{1,\ldots,f\}$$ and suitable translates $P'_i$ of $P_i$ with the property that
$$K\setminus \cup_{i=1}^j P'_{\sigma(i)}$$
is a \emph{connected} set for $j=1,\ldots, f$, and is empty for $j=f$.
\end{defi}

In the present paper, we study the following variant of the problem addressed by Conjecture~\ref{conj1}: What is the smallest integer $f=f(d,\epsilon)$ such that every collection of $f$ planks of width $\epsilon$ in $\R^d$ permits a non-dissective translative covering of the unit ball $\R^d$?
\smallskip

As we have seen before, we have $f(2,\epsilon)=\Theta(1/\epsilon)$ and, for every $d\ge 3,$
\begin{equation}\label{EQ}
 c_d\epsilon^{-1}\le f(d,\epsilon)\le C_d\epsilon^{-(d+1)/2}.
\end{equation}

Here, we concentrate on the first open case $d=3$. Our main theorem improves on both the lower and the upper bound in (\ref{EQ}).

\begin{thm} There exist constants $c,C>0$ such that  $$c \epsilon^{-4/3}\le f(3,\epsilon)\le C\epsilon^{-7/4}.$$ \label{thm1}
\end{thm}
In particular, the lower bound shows that Conjecture~\ref{conj1} does not remain true if we restrict our attention to non-dissective coverings.

In the next two sections, we establish the lower and upper bounds in Theorem ~\ref{thm1}, respectively. We conclude this note with a few remarks (see Section 5).


\section{Proof of Theorem~\ref{thm1} -- Lower bound}

In this section, we establish the lower bound in Theorem ~\ref{thm1}. Consider $k = \Omega(\epsilon^{-4/3})$ points on the sphere such that the angular distance between any two of them is at least $10\epsilon^{2/3},$ and each point is within angle $1^\circ$ from the north pole.

Let $P_1,\ldots, P_k$ be a sequence of planks of width $\epsilon$, whose normal vectors correspond to the above $k$ points. Consider a non-dissective covering of the unit ball $B^3$ using these planks, in exactly this order. Since all planks are almost horizontal, we may talk about their `upper' and `lower' boundary planes.
\smallskip

Let $K_0=B^3$ and, for all $i,\; 1\le i \le k,$ let $K_i:=K_{i-1}\setminus P_i'$, where $P'_i$ is a suitable translate of $P_i$ used for the covering. By the assumption that the covering is non-dissective, $K_i$ is a connected set for every $i$. As soon as we cover a point of $B^3$ within angular distance $1^\circ$ from the equator, we stop the covering process, and we declare that the covering was \emph{successful}. We want to show that the number of planks used up to this moment is $\Omega(\epsilon^{-4/3}).$

We may assume recursively that $P'_{i+1}$, the translate of $P_{i+1}$ that we use for the covering, satisfies one of the following two conditions
\begin{enumerate}
    \item either the upper boundary plane of $P'_{i+1}$ coincides with the upper supporting plane of $K_i$,
    \item or the lower boundary plane of $P'_{i+1}$ coincides with the lower supporting plane of $K_{i}$.
\end{enumerate}
Assume without loss of generality that the first plank that covers a point at distance at most $1^\circ$ from the equator satisfies condition 1 above. Note that this implies that at this moment we have already covered a portion of the upper half of $B^3$ whose volume is $\Omega(1)$. Slightly abusing notation, we denote the planks that were used for the covering of the upper hemisphere also by $P'_1,\ldots, P'_k.$
\smallskip

We can bound the volume of $K_{i}$ that is covered by $P'_{i+1}$: it is at most $\epsilon$ times the area of the {face} 
of $K_{i+1}$ lying on the boundary of the plank  $P'_{i+1}$.  Note that this face is the intersection of the lower boundary plane of $P'_{i+1}$ with $K_i$.
Let $X$ denote this intersection.

\begin{lem}\label{lem1}
    $X$ does not contain a disk of radius $10\epsilon^{1/3}$.
\end{lem}
The lemma implies the lower bound as follows. The diameter of $X$ is at most $2$, which is the diameter of $B^3$. Combining this fact with Lemma~\ref{lem1}, we obtain that the area of $X$ is $O(\epsilon^{1/3})$. This, in turn, implies that, for every $i,\; 1\le i\le k,$ the volume of the intersection $P'_{i+1}\cap K_i$ is $O(\epsilon^{4/3})$. The volume of the upper part of $B^3$ that we need to cover is $\Omega(1)$. Hence, we need $\Omega(\epsilon^{-4/3})$ planks to exhaust this volume.

\begin{figure}[h]
\includegraphics[width=\textwidth]{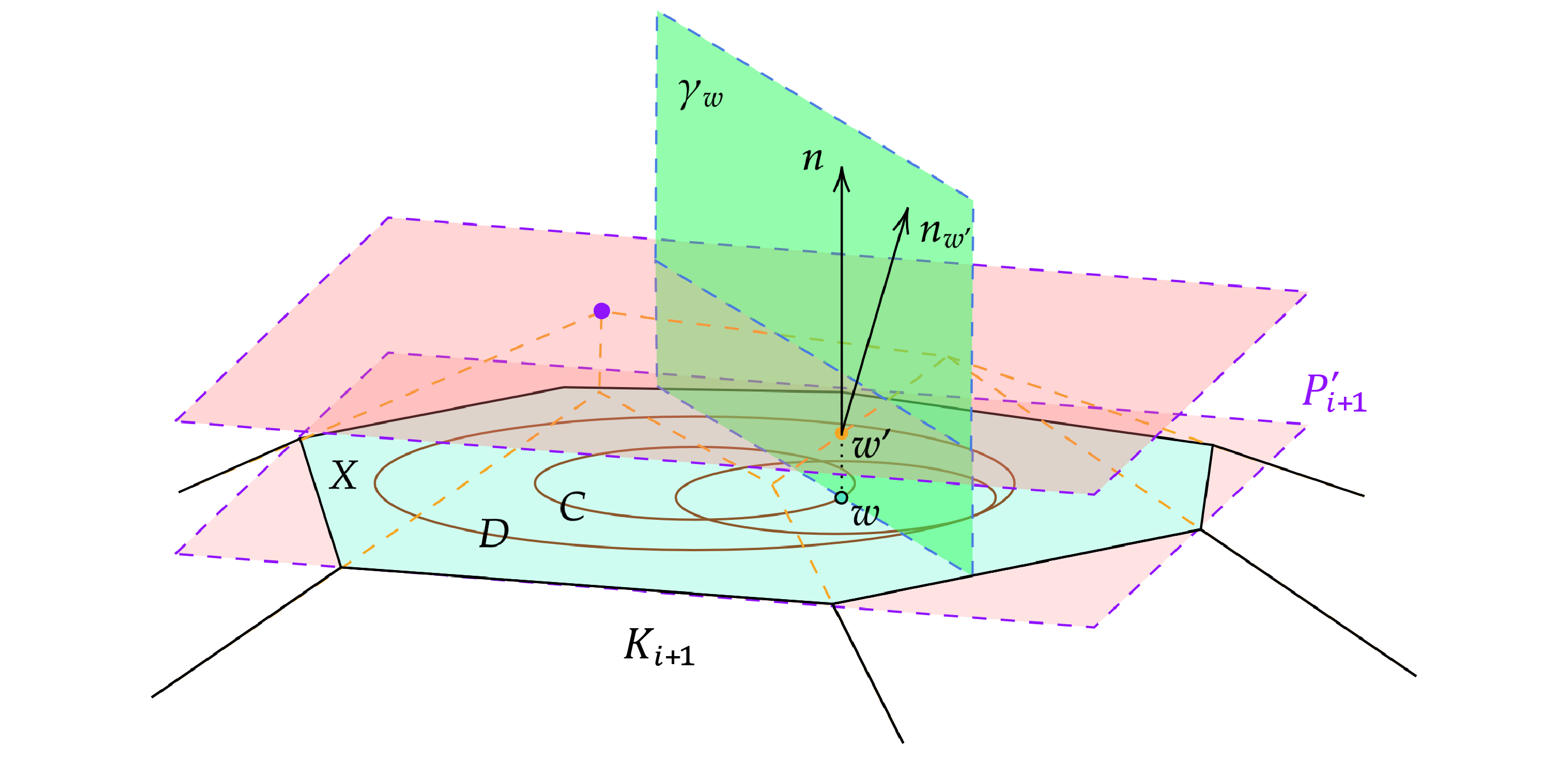}
\caption{Discs $D,C$ on the face $X$ of the set $K_{i+1}$. The plank $P'_{i+1}$ is shown by two parallel planes in semi-transparent red. The part of $K_{i}$ that was cut off by $P'_{i+1}$ is indicated in dashed orange. the plane $\gamma_w$ containing $w,w'$ and normal vectors $n, n_{w'}$ is semi-transparent green.} \label{fig1}
\end{figure}

\begin{proof}[Proof of Lemma~\ref{lem1}] Let $n$ stand for the normal vector to $\pi^L$. For any point $w$ in $X$, we denote by $w'$ the point on the boundary of $K_i$ that is obtained from $w$ by translating it parallel to $n$ inside $P'_{i+1}$. Note that $w'$ and $w$ are at most $\epsilon$ distance apart.

Let us show that $X$ does not contain a disk $D$ of radius $10\epsilon^{1/3}$.  See Figure~\ref{fig1}. Assume the contrary, and let $v$ be the center of such a disk $D$.

Consider the circle $C$ in $\pi^L$, concentric to $D$ and of radius $5\epsilon^{1/3}$. Take a point $w\in C$. Consider  the normal vector $n_{w'}$ of a supporting hyperplane for $K_i$ at $w'$, which choice we explain below. If the boundary of $K_i$ at $w'$ has a flat part, then it must also lie on the boundary of one of the planks $P_{j}'$, for some $j\le i$. Then take as $n_{w'}$ the normal vector to $P_{j}'$. If the boundary of $K_i$ at $w'$ has no flat part, then it must lie on the boundary of the initial ball $B^3$, and then we simply take the unique supporting hyperplane to the ball at $w'$.

Consider the intersection $Q$ of $K_i\cap P'_{i+1}$ with the plane $\gamma_{w}$ spanned by $n$ and $n_{w'}$ and passing through $w$. By our assumption, $X$ contains a disc of radius $5 \epsilon^{1/3}$ centered at $w$.
Thus, $Q$ a convex set of height at most $\epsilon$ and with the `base' being a straight line segment orthogonal to $n$ that passes through $w$ and has length  at least $5\epsilon^{1/3}$ in each direction from $w$. Thus, the angle between $n$ and $n_{w'}$ is at most 
$\epsilon^{2/3}$: otherwise, the supporting line at $w'$ corresponding to $n_{w'}$ will hit the base segment at distance at most $\epsilon\cdot \cot(\epsilon^{2/3})<2\epsilon^{1/3}$ from $w$, which is a contradiction. See Figure~\ref{fig2}.

\begin{figure}[h]
\includegraphics[width=\textwidth]{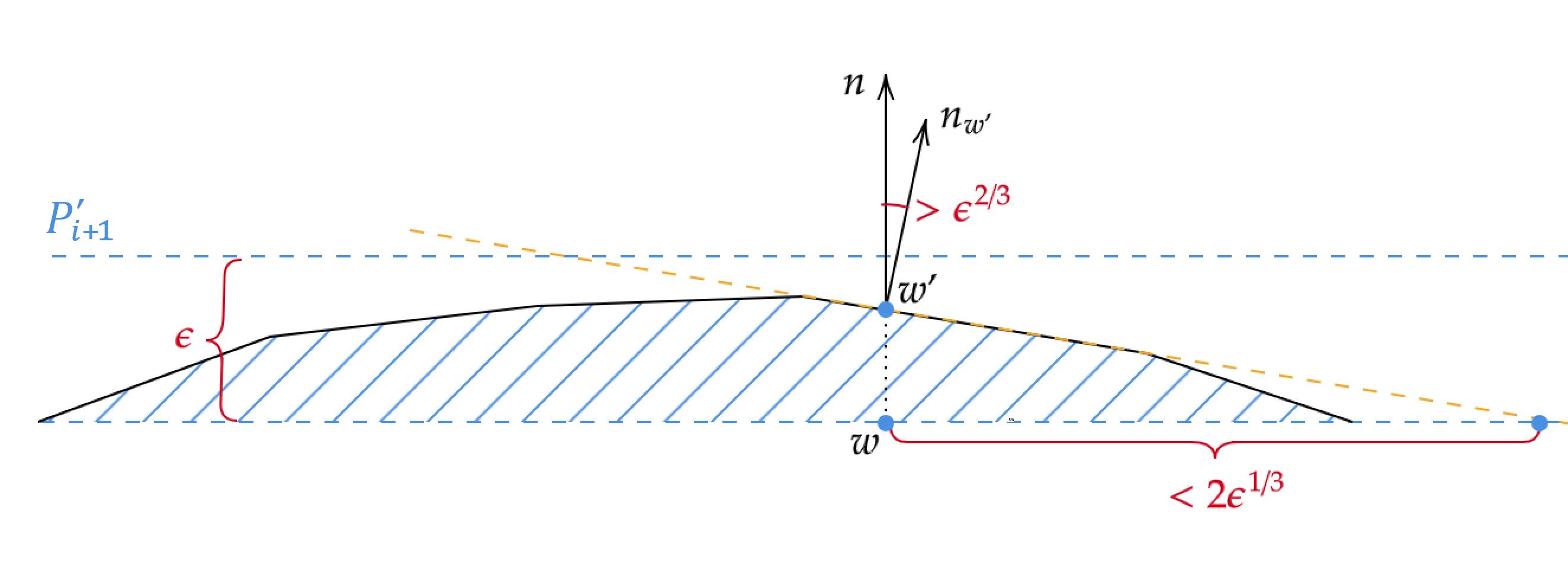}
\caption{An illustration to the proof of Lemma~\ref{lem1}: If the angle between $n$ and $n_{w'}$ is bigger than $\epsilon^{2/3}$, then the supporting line at $w'$ hits the `base' $\pi^L\cap \gamma_w$ at distance at most $2\epsilon^{1/3}$. } \label{fig2}
\end{figure}

We conclude that 
$n_{w'}$ has angle at most $\epsilon^{2/3}$ with $n$ for any $w\in C$. Since the angle between the normal vectors of any two planks is at least $10\epsilon^{2/3}$,  $n_{w'}$ cannot be a normal vector to any other plank. Running the same argument for each point $w\in C$, We conclude that the set $C':=\{w':w\in  C\}$ lies on the boundary of $B^3$. 
But then the angle between the normal vectors at antipodal points on  $C'$ is $\Omega(diam(C)) = \Omega(\epsilon^{1/3})$, a contradiction with the fact that each of them has angle at most $\epsilon^{2/3}$ with $n$. 
\end{proof}

 \section{Proof of Theorem~\ref{thm1} -- Upper bound}

Consider $k$ planks of width $\epsilon$, and write $k$ in the form $k = \epsilon^{-\beta}$ for some $\beta>0.$ Let $\alpha>0$ be a constant to be specified later. By averaging over all directions, we can find a vector $n$ such that there are  $\Theta(\epsilon^{2\alpha-\beta})$ planks whose normal vectors form an angle of at most $\epsilon^\alpha$ with $n$. Let these planks be $P_1,\ldots, P_m$.
(At the end of the proof we will see that for our arguments to work, the optimal choice of the parameters will be $\beta\approx 7/4$ and $\alpha\approx 3/4$.)
\smallskip

Choose a coordinate system in which the direction of the positive $z$-axis, pointing upwards, is $n$. Since the normal vectors of the planks $P_1, \ldots, P_m$ are very close to $n$, we can call these planks `almost horizontal,' and we can talk about their `upper' and `lower' boundary planes.
\smallskip

Given a convex body $K$, the \emph{(outer) parallel body} of $K$ at distance $1$ is defined as $$B(K):= \{v\in \R^3: \inf_{w\in K}\|v-w\|\le 1\}.$$
Clearly, $B(K)$ is a smooth body, hence in every direction it has precisely two supporting planes.

\smallskip

We are going to process the planks one by one. Set $K_0=B^3$. Let $P_1'$ denote the translate of $P_1$ whose upper boundary plane is tangent to $K_0$ from above. Informally, this means that we move $P_1$ from above as long as we can, preserving the property that $K_0\setminus P_1'$ remains connected. Since $K_0=B^3$, we have $B(K_0)=2B^3$. Denote by $P_1''$ the translate of $P_1$ whose upper boundary plane is tangent from above to $B(K_0)$.

Next, we set $K_1:=K_0\setminus P_1'$. Let $P_2'$ denote the translate of $P_2$ whose upper boundary plane is tangent to $K_1$ from above. Analogously, denote by $P_2''$ the translate of $P_2$ whose upper boundary plane is tangent to $B(K_1)$ from above. Let $K_2:=K_1\setminus P_2$ Proceeding like this, after performing $j$ steps ($j\le m$), we obtain a sequence of convex bodies $K_0=B^3\supset K_1\supset K_2\supset\cdots\supset K_j,$ and two sequences of planks,  $P_1',\ldots,P_j'$ and $P_1'',\ldots,P_j''$ with the following properties:
\begin{enumerate}
\item  $B(K_0)\supset B(K_1)\supset\cdots\supset B(K_j)$
\item  For every $i,\;1\le i\le j,$ the distance between the lower (resp. upper) boundary planes of $P_i'$ and $P_i''$ is $1$.
\item  If $K_{j}=K_{j-1}\setminus P_j'$ is \emph{empty}, we terminate the process.
\end{enumerate}

By definition, for every $j$, we have $K_{j}=K_0\setminus \cup_{i=1}^j P'_i.$  We will keep track of the (decreasing) sequence of volumes $\Vol B(K_i)$. If for some $j,$ we have
\begin{equation}\label{eqvolbound}\Vol B(K_{j}) < \Vol B^3,\end{equation}
then we can conclude that $K_{j}$ is empty, i.e., the initial ball $B^3$ is completely covered by the planks $\cup_{i=1}^{j}P'_i.$
\smallskip

We need some easy observations.

\begin{prop}\label{prop1} Let $\pi_i', \pi_i''$ be the lower boundary planes of $P_i',P_i''$, respectively, and let $\Pi_i$ be the upper half-space bounded by the $\pi''_i$.

Then, for $i=1,2,\ldots$, we have $$B(K_{i})=B(K_{i-1}\setminus P_i')\subset B(K_{i-1})\setminus \Pi_i\subset B(K_{i-1})\setminus P_i''.$$
\end{prop}
\begin{proof}
We start by verifying the first inclusion. Let $v$ be any point of $B(K_{i-1}\setminus P_i')$. By definition, there exists $w\in K_{i-1}\setminus P_i' $    with $||v-w||\le 1.$ Thus, $v$ also belongs to $B(K_{i-1})$. We only have to argue that $v$ does not lie inside the half-space $\Pi_i.$  If it did, by property 2 above, it would be farther away than $1$ from the lower boundary plane of $P_i'$, which separates $K_{i-1}\setminus P_i'$ from $v$. This would contradict our assumption that $v\in B(K_{i-1}\setminus P_i').$

The second inclusion holds since $P''_i\subset \Pi_i.$
\end{proof}

Recall that, according to our conventions, the direction of the positive $z$-axis of the coordinate system coincides with the vector $n$. At the beginning, we fixed a vertical direction $n$, but in the sequel we will also use the following simple fact in other scenarios.

\begin{prop}\label{prop2}
Let $K$ be a convex body and let $s$ be a boundary point of the parallel body $B(K)$ such that the $z$-coordinate of $s$ is maximum.

Then $B(K)$ contains a ball $U$ of unit \emph{radius} that touches the boundary of $B(K)$ at $s$. \qed
\end{prop}


Now we quickly outline our proof strategy for the upper bound in Theorem~\ref{thm1}. Recall that we selected a `dense' subset of almost parallel directions, in terms of two constants $\alpha$ and $\beta$. Our first and main goal is to bound from below the volume of the parallel body $B(K_j)$ of the part of $B$ that we cover with the first $j$ planks corresponding to the chosen directions. To get this lower bound, we will have two cases to consider, and the parameter $\alpha$ is chosen to balance the bound in the two cases. We cover a substantial part of $K_0=B^3$ using these planks, and remove them. This constitutes the first stage of our covering algorithm.

At the second stage, we repeat the same procedure. This time, in the remaining set of directions we find a `dense' subset (using the same $\alpha$), and we use the corresponding, almost parallel, planks to cover the remaining part of the ball. While at stage 1, we started with $K_0=B^3,$ at stage 2 we start where we left off at stage 1. That is, our new initial body $K^{(2)}_0$ will be identical with last $K_j$ at stage 1. Again, we bound from below the volumes of the parallel bodies of the new shrinking sequence
$$K^{(2)}_0\supset K^{(2)}_1 \supset K^{(2)}_2\supset\cdots$$
of uncovered bodies. Because of our choice of parameters, it is guaranteed that after performing this procedure a bounded number of times, at some stage we reach a point where \eqref{eqvolbound} holds. At this point, we can be sure that $B^3$ is entirely covered.
\smallskip

Apply Proposition~\ref{prop2} with $K_0=B$ to obtain a ball $U\subset B(K_0)$ of unit radius that touches the boundary of $B(K_0)$ at its highest point in the vertical direction $n$. Let $\ell$ be the vertical line passing through the center of $U$.

We distinguish two cases.
\smallskip

\textbf{Case A:} There is a set of indices $A\subset \{1,\ldots,m\}$ with $|A|\ge m/2$ such that for any $i\in A,$ the intersection of $\ell$ with $B(K_i)$ is at least $\epsilon/2$ shorter than its intersection with $B(K_{i-1})$.
\smallskip

\textbf{Case B:} There is a set of indices $B\subset \{1,\ldots,m\}$ with $|B|\ge m/2$ such that for any $i\in B$, the intersection of $\ell$ with $B(K_i)$ is less than $\epsilon/2$ shorter than its intersection with $B(K_{i-1})$.

\smallskip

In \textbf{Case A}, we start our argument as follows. To simplify the notation, assume that $m\in A$. Suppose that the center of $U$ is inside $B(K_m)$ or we have $m\epsilon = o(1)$. Using Proposition~\ref{prop1}, we obtain that
\begin{equation*}B(K_m)\subset B(K_{m-1})\setminus \Pi_m \subset B(K_0)\setminus \Pi_m.
\end{equation*}
      Since $U\subset K_0$, we also have that
      $$U\cap \Pi_m\subset  B(K_0)\cap \Pi_m.$$
      By the above inclusions, we can bound from below the difference between the volumes of $B(K_m)$ and  $B(K_0)$.
      \begin{align*}\Vol B(K_0)-\Vol B(K_m)&\ge
      \Vol B_0 - \Vol \big(B(K_0) \setminus \Pi_m\big)\\
      &= \Vol\big(B(K_0)\cap \Pi_m)\big) \\
      &\ge \Vol\big(U\cap \Pi_m\big)=\Omega(m^2\epsilon^2).\end{align*}

We explain the last equality. The set $U\cap \Pi_m$ is a spherical cap that contains a section of the vertical line $\ell$ of length at least $m\epsilon/4$. 
The lower boundary plane $\pi_m''$ of $P''_m$ is almost horizontal: its normal vector has an angle of at most $\epsilon^{\alpha}$ with $\ell$. Thus, $U\cap \Pi_m$ contains a spherical cap and, hence, a cone whose height is roughly the same as as the length of the portion of $\ell$ contained in $U\cap \Pi_m$, which is  $\Omega(m\epsilon)$. The radius of the circular basis of this cone (and of the cap) is $\Omega(\sqrt{m\epsilon})$. The volume of this cone is $\Omega(m^2\epsilon^2),$ as claimed.


    If the center of $U$ is outside $Q_m$ or if $m\epsilon = \Omega(1)$, 
    then we have
    $$ \Vol B(K_0)-\Vol B(K_m)=\Omega(1).$$
Recall that $m  = \Theta(\epsilon^{2\alpha-\beta})$. Thus, the volume that we managed to cover in this case is at least
\begin{equation}\label{eqbound1}\min\big\{\Omega\big(m^2\epsilon^2\big),\Omega(1)\big\} = \min\big\{\Omega\big(\epsilon^{4\alpha-2\beta+2}\big),\Omega(1)\big\}.\end{equation}

In \textbf{Case B}, we proceed as follows. In view of Proposition~\ref{prop1}, the condition implies that for any $i\in B$, the intersection of $\ell$ with $B(K_{i-1})\setminus \Pi_i$ is less than $\epsilon/2$ shorter than that of $B(K_{i-1}).$
\smallskip

\begin{figure}[h]
\includegraphics[width=\textwidth]{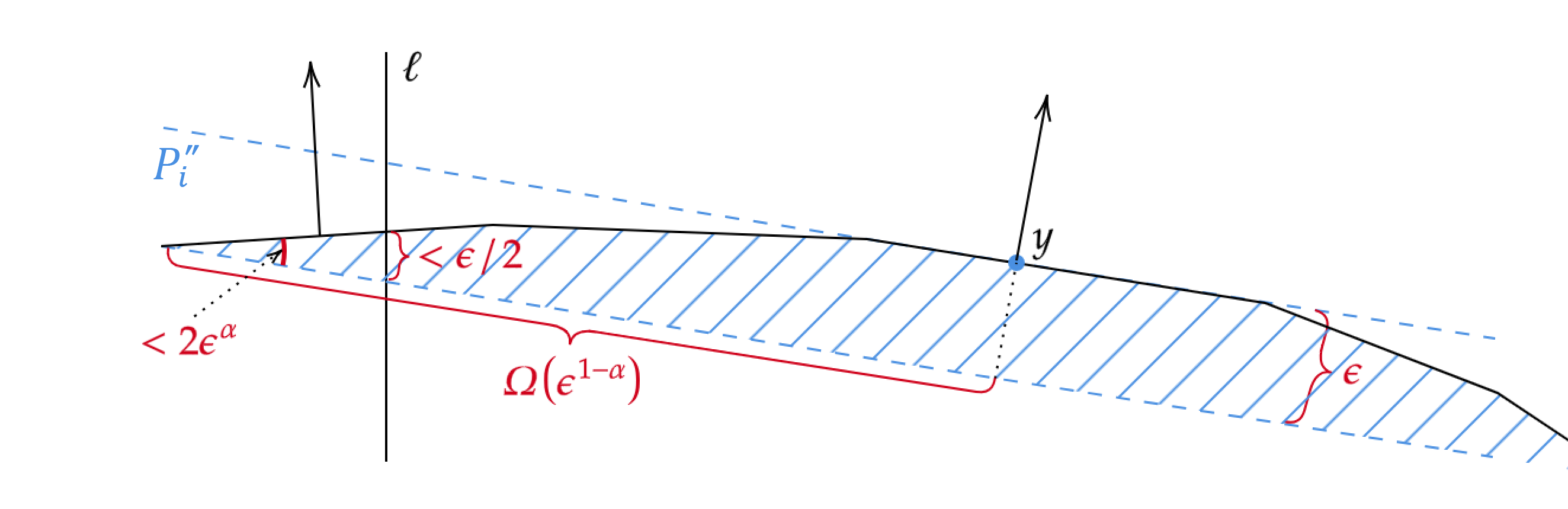}
\caption{An illustration for the proof of the upper bound in Case B. The convex curve is a portion of the boundary of the intersection of $B(K_{i-1})$ with $\gamma$. The shaded area is the part of this intersection covered by the plank $P_i''$, denoted by $
Q$. } \label{fig3}
\end{figure}

Consider the plane $\gamma$ spanned by $\ell$ and a point $y$ at which the upper boundary  plane of $P_i''$ touches $B(K_{i-1})$. Set $$Q:=\gamma\cap B(K_{i-1})\cap P_i''.$$
See Figure~\ref{fig3}. If $Q$ intersects $\ell$, then let $q$ denote the intersection point of the boundary of $Q$ with $\ell$. If $Q$ does not intersect $\ell$, then let $q$ be the point of $Q$ closest to $\ell$, i.e., the endpoint of the `base' of $Q$ closer to $\ell$.
 \smallskip

      In both cases, the angle between the normal to the boundary of $\gamma\cap B(K_{i-1})$ at $q$ and the normal at $y$ is at most $2\epsilon^\alpha$, because normal vectors of all planks $P_1,\ldots, P_m$ are within an angle of at most $\epsilon^\alpha$ from the vertical direction $n$ (which is the same as the direction of $\ell$). The \emph{height} of $Q$ at $y$, i.e., the intersection of $Q$ with a vertical line through $y$ is at least $\epsilon$, while its height at $z$ is at most $\epsilon/2$, by the assumption of Case B. 
      This implies that the base of $Q$ has length $\Omega(\epsilon^{1-\alpha})$.
\smallskip

      Stepping out of the plane $\gamma$, we find that the intersection of $P''_i$ with $B(K_{i-1})$ cuts out a whole 2-dimensional face. This face contains a disk of radius $\Omega(\epsilon^{1/2})$. Indeed, as before, $B(K_{i-1})$ contains a unit ball that is tangent to the boundary of $B(K_{i-1})$ at the point $y$. Combining the last two facts, we obtain that the area of the face cut out by $P''_i$ is $\Omega(\epsilon^{1-\alpha})\cdot\Omega(\epsilon^{1/2}) = \Omega(\epsilon^{3/2-\alpha})$. Thus, $B(K_{i-1})\cap P''_i$ contains a cone with base of area $\Omega(\epsilon^{3/2-\alpha})$ and height $\epsilon$. This implies that
      $$\Vol (B(K_{i-1})\cap P''_i) = \Omega(\epsilon^{5/2-\alpha}).$$
      Therefore, in Case B, we have that
      \begin{equation}\label{eqbound2}\Vol B(K_0)-\Vol B(K_m)\ge \Omega(m\epsilon^{5/2-\alpha})=\Omega(\epsilon^{5/2+\alpha-\beta}).\end{equation}
\medskip

To complete the proof of the upper bound in Theorem~\ref{thm1}, we iterate the above procedure, as follows. After finishing the above procedure, stage 1 of our algorithm, we
remove all planks we have used so far. At the next stage, we find a new vector $n'$ such that among the remaining planks there are $\Theta(\epsilon^{2\alpha-\beta})$, the normal vectors of which enclose an angle of at most $\epsilon^\alpha$ with $n'$.  Repeating the same argument $\Omega(\epsilon^{-2\alpha})$ times, the total number of remaining planks is still at least $k/2$. Thus, up to constant factors, at each stage we obtain the same lower bound as above
for the total volume of the pieces cut off by the planks.
\smallskip

Combining \eqref{eqbound1} and \eqref{eqbound2}, we conclude that this procedure allows us to cover at least
$$\Omega(\epsilon^{-2\alpha})\cdot \min\big\{\epsilon^{2+4\alpha-2\beta},1, \epsilon^{2\alpha-\beta+5/2-\alpha}\big\} =  \Omega\big(\min\big\{\epsilon^{2+2\alpha-2\beta},\epsilon^{-2\alpha}, \epsilon^{5/2-\beta-\alpha}\big\}\big)$$
of the total volume of $B(B^3)$.
 Choose $\alpha$ such that $2+2\alpha-2\beta = 5/2-\alpha-\beta$, that is, $\alpha = \frac 16+\frac 13\beta$. Then the above minimum is equal to $\Omega(\min\{\epsilon^{\frac 73-\frac 43\beta},$ $\epsilon^{-\frac 13-\frac 23\beta}\})$. In view of (\ref{eqvolbound}), if this value is bigger than the volume of $B(B^3)$, then we must have covered the entire ball $B^3$, and we are done. If we put $\beta=\frac 74$, then this quantity is $\Omega(1)$. This means that taking $k=C \epsilon^{-7/4}$ with a sufficiently large $C$ is enough to cover $B^3$.

This completes the proof of the theorem.  \qed

\section{Concluding remarks}
\textbf{1.} In the present note, we focused on the non-dissective covering problem in 3 dimensions. For $d>3$, the same lower bound remains valid: indeed, in order to cover a ball in $\R^d$, one needs to cover any of its $3$-dimensional sections. One may try to go into the proof and adapt the lower bound argument to the $d$-dimensional case. We should then take $\Omega(\epsilon^{-2(d-1)/d})$ points on the sphere with pairwise distances at least $10\epsilon^{2/d}$. The same analysis shall lead to a slightly stronger lower bound $\epsilon^{-2(d-1)/d}.$ The exponent in the bound does not grow linearly with $d$, because Lemma~\ref{lem1} guarantees only one `short' direction in the intersection of a plank and the remaining body.
\smallskip

\textbf{2.} The proof of the upper bound can also be extended to $d>3$, and the resulting bound would be $C_d\epsilon^{-\frac{d+1}2 +\frac{d-1}{d^2-d+2}}$, which is better than \eqref{EQ}, albeit by an additive term in the exponent that tends to $0$ as $d$ grows. We believe that neither the lower nor the upper bound is tight.
\smallskip

\textbf{3.} Our upper bound argument in Section~\ref{sec3} can be easily turned into a polynomial time algorithm for constructing a cover. Take $k=C\epsilon^{-7/4}$ planks of width $\epsilon$. We first iteratively construct an ordering of the planks which will turn out to guarantee that we end up with a cover. For this, first we (approximately) compute the angle between any two normal vectors and construct a graph $G$ whose vertices are the selected $k$ planks. Two vertices are joined by an edge if the angle between the normal vectors of the corresponding planks is smaller than $\epsilon^{\alpha} = \epsilon^{3/4}$. Then, at each step, we
\begin{itemize}
    \item[ (a)] find a vertex $v$ whose degree is at least the average degree in $G$;
    \item[  (b)] include $v$ and its neighbors in the ordering (the order within this set is arbitrary);
    \item[   (c)] modify $G$ by deleting $v$ and its neighbors from $G$.
\end{itemize}
In the proof of the upper bound, at each step, we actually work with a chunk of vectors that were included in the ordering in one step.

Finally, we construct the covering, i.e., we find suitable translates of the planks that cover the ball, as follows. For the $i$-th plank in the ordering, we need to find the maximum (or an approximate maximum) of a linear function over the convex set $K$ that remains from the original ball after removing all points covered by the first $i-1$ planks. Then we modify the set $K$ by including an extra linear inequality in its definition, and repeat the whole procedure, if necessary.

\end{document}